\documentclass[final]{siamltex}
\usepackage[T1]{fontenc}

\usepackage{amssymb,latexsym,amsmath}
\usepackage{mathrsfs}
\usepackage{enumitem}
\usepackage{caption}
\usepackage{comment}
\usepackage{inputenc}
 \usepackage{makeidx}
\usepackage{color}

\usepackage{epsfig} 
\usepackage{verbatim}

\allowdisplaybreaks

\newtheorem{assumption}{Assumption}[section]
\newtheorem{example}{Example}[section]

\allowdisplaybreaks

\newcommand{\sy}[1]{{\color{black} #1}}

\usepackage{amsfonts}
\usepackage{float}
\usepackage{multirow}
\title{Comparison of Information Structures for Zero-Sum Games and a Partial Converse to Blackwell Ordering in Standard Borel Spaces\thanks{A conference version \cite{HYISIT2020} was presented at the International Symposium on Information Theory, as an announcement of the narrower versions of some of the results. This research was supported in part by
the Natural Sciences and Engineering Research Council (NSERC) of Canada.}}

\author{Ian Hogeboom-Burr \and Serdar Y\"{u}ksel\thanks{The authors are with the Dept. of Mathematics and Statistics, Queen's University, Kingston K7L 3N6, ON, Canada, {\tt\small \{15ijhb,yuksel\}@queensu.ca}.}}

\begin{document}
\maketitle

\begin{abstract}
In statistical decision theory involving a single decision-maker, an information structure is said to be better than another one if for any cost function involving a hidden state variable and an action variable which is restricted to be conditionally independent from the state given some measurement, the solution value under the former is not worse than that under the latter. For finite spaces, a theorem due to Blackwell leads to a complete characterization on when one information structure is better than another. For stochastic games, in general, such an ordering is not possible since additional information can lead to equilibria perturbations with positive or negative values to a player. However, for zero-sum games in a finite probability space, P\k{e}ski introduced a complete characterization of ordering of information structures. In this paper, we obtain an infinite dimensional (standard Borel) generalization of P\k{e}ski's result. A corollary is that more information cannot hurt a decision maker taking part in a zero-sum game. We establish two supporting results which are essential and explicit though modest improvements on prior literature: (i) a partial converse to Blackwell's ordering in the standard Borel  setup and (ii) an existence result for equilibria in zero-sum games with incomplete information.
\end{abstract}


\section{Introduction}\label{intro}

Characterizing the value of information structures is a problem in many disciplines involving decision making under uncertainty. In stochastic control theory, it is well-known that more information cannot hurt a given decision maker since the decision maker can always choose to ignore this information. In statistical decision theory involving a single decision maker, one says that an information structure is better than another one if for any given measurable and bounded cost function involving a hidden state variable and an action variable which is restricted to be only a function of some measurement, the solution value obtained under optimal policies under the former is not worse than the value obtained under the latter. For finite probability spaces, Blackwell's celebrated theorem \cite{Blackwell} on the ordering of information structures obtains a precise characterization of when an information structure is better. This finding has inspired much further research as reviewed in e.g. \cite{LeCamReview, raginsky2011shannon}.

Since Blackwell's seminal 1953 paper \cite{Blackwell}, significant work has been done to extend Blackwell's results to team problems and games. Stochastic team problems (known also as identical interest games) were studied in a finite-space setting by Lehrer, Rosenberg, and Shmaya \cite{Lehrer}; see also \cite[Chapter 4]{YukselBasarBook}. The value of information in various types of repeated games has also been explored in \cite{Kandori1991}, \cite{KandoriObara2006}, and \cite{Kloosterman2015}. 

In general games, information can have both positive and negative value to a player since additional information can lead to a perturbation which is not necessarily monotone due to the presence of competitive equilibrium, unlike in a team setup. Some of the earlier accounts on such phenomena are \cite{hirshleifer1971private} and \cite{basar1974informational}, where the latter studied the comparison of information structures for team-like (LQG) and zero-sum like (quadratic duopoly) games. 

As noted above, for general non-zero sum game problems, informational aspects are very challenging to address and more information can hurt some or even all of the players in a system, see e.g. \cite{hirshleifer1971private,gossner2001value,kamien1990value,tBasarStochasticDiffGames}. To make this discussion more concrete, we provide the following example due to Bassan et al. \cite{Bassan2003}. 

\begin{example}
Consider a card drawn at random from a deck, where its colour can be either red or black, each with probability $1/2$. Player 1 first declares his guess of the colour, and then, after hearing what Player 1 guessed, Player 2 submits her guess for the colour. If both players guess the same colour, the payout is \$2 each, whereas if one player guesses correctly, that player receives a payout of \$6 and the other player receives \$0. 

In the case where both players are uninformed about the colour of the card, the expected payout is \$3 each, as Player 1's optimal strategy is arbitrary, and Player 2's optimal strategy is to guess the opposite colour of what Player 1's guessed.

In the case where both players are informed of the colour of the card prior to declaring their guess, the equilibrium for the game occurs when both players guess the true colour of the card. In this case, the expected payout becomes \$2 for each player.  \hfill $\diamond$
\end{example}


Bassan et al. further provided sufficient conditions for games to have the `positive value of information property', where providing additional information to some or all players results in greater or equal payoffs for all players \cite{Bassan2003}. Gossner and Mertens highlighted zero-sum games as a particularly interesting class to study in the context of ordering information structures in games and did preliminary work on this ordering \cite{gossner2001value}; zero-sum games provide a worthwhile class of games to study due to the fact that, under mild conditions, every game has a value (achieved at a saddle point). 

For comparison of information structures in zero-sum games with finite measurement and action spaces, P\c{e}ski provided necessary and sufficient conditions, and thus a complete characterization \cite{Peski}. Prior to P\c{e}ski's results, De Meyer, Lehrer, and Rosenberg had shown that the value of information is positive in zero-sum games, albeit with a slightly different setup than P\c{e}ski, where their payoff depended on an individual `type' for each player rather than a common state of nature; their results were applicable for infinite action spaces and finite type spaces \cite{Lehrer2009}. Furthermore, Lehrer and Shmaya studied a `malevolent nature' zero-sum game played between nature and a player in a finite setting, and characterized a partial ordering of information structures for these games \cite{LehrerShmaya2008}. We also note the following references on topological and continuity properties of information structures in single-agent and multi-agent team problems \cite{YukselOptimizationofChannels,YukselWitsenStandardArXiv,YukselSaldiSICON17}. A recent comprehensive study on the value and topological properties of information structures in zero-sum games, which also generalizes \cite{Peski} to the countably infinite probability space setup, is \cite{pkeski2019value}. 

In this paper, we generalize P\k{e}ski's results to a broad class of zero-sum games with standard Borel measurement and action spaces: we recall that a metric space which is complete and separable is called a Polish space, and a Borel subset of a Polish space is called a {\it standard Borel} space. Finite dimensional real vector spaces are important examples of such spaces. 

Toward this goal, additional supporting results, which may be of independent interest, are obtained: sufficient conditions are presented (i) for the existence of saddle-point equilibria in zero-sum games with incomplete information and (ii) for a partial converse to Blackwell's ordering when the player has standard Borel measurement and action spaces and the unknown variable also takes values from a standard Borel space.

\section{A review of prior results and contributions}

\subsection{Comparison of information structures in single-agent problems}
Let $x \sim \zeta$ be an $\mathbb{X}$-valued random variable with $\mathbb{X}$ being a standard Borel space. We call $x$ the state of nature; $\zeta$ is known by the decision maker but $x$ is not. Recall that a standard Borel space is a Borel subset of a complete, separable, metric (Polish) space. Let $\mathbb{Y}$, our measurement space, be another standard Borel space and $y$ be $\mathbb{Y}$-valued, defined with
\[y = g(x, \omega),\] for some independent noise variable $\omega$ (which, without any loss, can be taken to be $[0,1]$-valued).  In the above, we can view $g$ as inducing a measurement channel $Q$, which is a stochastic kernel  or a regular
conditional probability measure from  $\mathbb{X}$ to $\mathbb{Y}$ in the sense that
$Q(\,\cdot\,|x)$ is a probability measure on the (Borel)
$\sigma$-algebra ${\cal B}(\mathbb{Y})$ on $\mathbb{Y}$ for every $x
\in \mathbb{X}$, and $Q(A|\,\cdot\,): \mathbb{X}\to [0,1]$ is a Borel
measurable function for every $A \in {\cal B}(\mathbb{Y})$.

Given a fixed $\mathbb{X}$, $\mathbb{Y}$, and $\zeta$, a single player decision problem is a pair $(c, \mathbb{U})$ of a cost function $c: \mathbb{X} \times \mathbb{U} \rightarrow \mathbb{R}$ and an action space $\mathbb{U}$. 

Using stochastic realization results (see Lemma 1.2 in \cite{gihman2012controlled}, or Lemma 3.1 of \cite{BorkarRealization}), it follows that the functional representation in $y=g(x,v)$ is equivalent to a stochastic kernel description of an information structure, since for every $Q$, one can define $g$ and a $[0,1]$-valued random function $V$ so that the representation holds almost surely.

Let $\mathcal{P}(\mathbb{X})$ denote the set of all probability measures on (the Borel sigma field over) $\mathbb{X}$. For $\sy{\zeta}\in  \mathcal{P}(\mathbb{X})$ and kernel $Q$, we let $\sy{\zeta} Q$ denote the joint distribution induced on
$(\mathbb{X}\times \mathbb{Y}, \mathcal{B}(\mathbb{X}\times
\mathbb{Y}))$ by channel $Q$ with input distribution $\sy{\zeta}$:
\[  \sy{\zeta} Q(A) = \int_{A} Q(dy|x)\sy{\zeta} (dx), \quad A\in  \mathcal{B}(\mathbb{X} \times \mathbb{Y}). \]


Now, let the objective be one of minimization of the cost
\begin{eqnarray}\label{Cost}
J(\sy{\zeta},Q,{\gamma})=E_{\sy{\zeta}}^{Q,{\gamma}}\bigg[c(x,u)\bigg],
\end{eqnarray}
over the set of all admissible measurable policies $\Gamma := \{\gamma: \mathbb{Y} \to \mathbb{U}\}$ with $u=\gamma(y)$, where
$c:\mathbb{X}\times \mathbb{U}\to \mathbb{R}$ is a Borel measurable cost function and $E_{\sy{\zeta}}^{Q,\gamma}$ denotes the expectation with initial state probability measure given by $\sy{\zeta}$, under policy $\gamma$, and given channel $Q$.

The comparison question is the following: when can one compare two measurement channels $Q^1, Q^2$ such that
\[ \inf_{{\gamma} \in {\bf \Gamma}} J(\sy{\zeta},Q^1,{\gamma}) \leq \inf_{{\gamma} \in {\bf \Gamma}} J(\sy{\zeta},Q^2,{\gamma}),
\]
for a large class of single-player decision problems in (\ref{Cost})?

We now recall the notion of garbling. We note that garbling is sometimes defined to be equivalent to physical degradedness of communication channels (as opposed to stochastic degradedness) \cite{Cover}, however in this paper we will take stochastic degradedness and garbling to be equivalent.

\begin{definition}\label{garbling}
An information structure induced by some channel $Q_2$ is garbled (or stochastically degraded) with respect to another one, $Q_1$, if there exists a channel $Q'$ on $\mathbb{Y}\times \mathbb{Y}$ such that
\[
Q_2(B|x) = \int_{\mathbb{Y}} Q'(B|y) Q_1(dy|x), \; B \in {\cal
  B}(\mathbb{Y}),\;\sy{\zeta} \; a.s. \;x \in \mathbb{X}.
\]
\end{definition}
\sy{We also define the notion of \textit{more informative than} and introduce a useful result:}
\begin{definition}\label{moreinformative}
An information structure $\mu$ is \textit{more informative than} another information structure \sy{$\nu$} if 
\begin{equation}
    \inf_{\gamma \in \Gamma} E^{\sy{\nu},\gamma}_\sy{\zeta}[c(x,u)] \geq \inf_{\gamma \in \Gamma} E^{\mu,\gamma}_\sy{\zeta}[c(x,u)]\sy{,}\nonumber
\end{equation}
for all single player decision problems $(c(x,u), \mathbb{U})$. 
\end{definition}
\sy{
\begin{proposition}\label{useful}
The function 
\[V(\zeta):= \inf_{u\in \mathbb{U}} \int c(x,u) \zeta(dx),\]
is concave in $\zeta$, under the assumption that $c$ is measurable and bounded.
\end{proposition}

For a proof of this proposition see \cite[Theorem 4.3.1]{YukselBasarBook}.

We emphasize that in Definition \ref{moreinformative}, $\mathbb{U}$ is also a design variable for the decision problem. For instance, if $\mathbb{U}$ were a singleton, then the comparison of information structures would be meaningless. With this in mind, and in view of Proposition \ref{useful}, we state Blackwell's classical result in the following.}

\begin{theorem} \label{BlackwellInformative} [Blackwell \cite{Blackwell}]
Let $\mathbb{X}, \mathbb{Y}$ be finite spaces. The following are equivalent:
\begin{itemize}
\item[(i)]
$Q_2$ is stochastically degraded with respect to $Q_1$ (that is, a garbling of $Q_1$).
\item[(ii)] The information structure induced by channel $Q_1$ is more informative than the one induced by channel $Q_2$ for all single player decision problems with finite $\mathbb{U}$.
\end{itemize}
\end{theorem}

That (i) implies (ii) \sy{for general spaces} follows from \sy{Proposition \ref{useful}}, which is an immediate finding in statistical decision theory,  and Jensen's inequality \sy{\cite[Theorem 4.3.2]{YukselBasarBook}.  We also note that this result will hold, and the proof will follow in an identical manner, if the player is allowed to use randomized policies, i.e. $u = \gamma(y, \omega)$, where $\omega$ is an independent noise variable.}

The converse, ii) implies i), is significantly more challenging. For the case with general spaces, related results are attributed to \cite{CBoll}, and \cite{cartier1964comparaison}, \cite{strassen1965existence}, which relate an ordering of information structures in terms of dilatations and their relation with comparisons under concave functions defined on conditional probability measures. A very concise yet informative review is in \cite[p. 130-131]{LeCamReview} and a more comprehensive review is in \cite{torgersen1991comparison}. We will present a direct proof that will be utilized in our main result of the paper and present a comparative discussion.

\subsection{Comparison of information structures in zero-sum game problems}

Now, consider a zero-sum game generalization of the problem above, with two decision makers. 

Consider a two-agent setup as follows.
\begin{eqnarray*}
y^i&=& g^i(x,v^i), \quad i=1,2,
\end{eqnarray*}
where the noise variables $v^1$ and $v^2$ are independent. Suppose that $g^i$ induces a channel $Q^i$ for $i=1,2$ as described earlier and DM $i$ has only access to $y^i$. Let ${\bf \underline{\gamma}} = \{\sy{{\gamma}^1,{\gamma}^2}\}$ denote the measurable policies of the agents. 

Given fixed $\mathbb{X}$, $\mathbb{Y}^1$, $\mathbb{Y}^2$, and $\zeta$ such that $x \sim \zeta$, a game \sy{$G = (c, \mathbb{U}^1, \mathbb{U}^2)$} is a triple of a measurable and bounded cost function \sy{$c: \mathbb{X} \times \mathbb{U}^1 \times \mathbb{U}^2 \rightarrow \mathbb{R}$} and action spaces for each player $\mathbb{U}^1, \mathbb{U}^2$.

We will impose one of the following conditions on the information structures. We note that Assumption \ref{infoStrucConditionSpecific} implies  Assumption \ref{infoStrucCondition}, but this assumption often allows for a simpler interpretation.  That this implication holds is a consequence of the independent measurements reduction formulation to be explained in detail later in the paper (see Theorem \ref{Equil}). The results will be presented under the more general Assumption \ref{infoStrucCondition}. 

\begin{assumption}\label{infoStrucCondition}
The information structure is absolutely continuous with respect to a product measure:
\[P(dy^1, dy^2,dx) \ll \bar{Q}^1(dy^1)\bar{Q}^2(dy^2)\sy{\zeta}(dx),\]
for reference probability measures $\bar{Q}^i$, $i=1,2$. That is, there exists an integrable $f$ which satisfies for every Borel $A, B, C$
\[P(y^1 \in B, y^2 \in C, x \in A) = \int_{A,B,C} f(x,y^1,y^2) \sy{\zeta}(dx)\bar{Q}^1(dy^1)\bar{Q}^2(dy^2)\sy{.}\]
\end{assumption}

\begin{assumption}\label{infoStrucConditionSpecific}
The following conditional independence (or Markov) condition holds: 
\begin{equation} P(dy^1, dy^2,dx) = Q^1(dy^1|x)Q^2(dy^2|x)\sy{\zeta}(dx)\sy{.}\nonumber \end{equation} 
where the measurements of agents are absolutely continuous so that for $i =1,2$, there exists a non-negative function $f^i$ and a reference probability measure $\bar{Q}^i$ such that for all Borel $S$:
    \begin{align*}
         Q^i(y^i \in S | x)  = \int_S f^i (y^i, x)\bar{Q}^i(dy^{i})\sy{.}\nonumber
    \end{align*}
\end{assumption}

%

Let the joint measure \sy{$P(dy^1, dy^2,dx)$} define the {\it information structure} for the game and let us denote this with $\mu$. For a zero-sum game with the conditional independence assumption in \sy{Assumption \ref{infoStrucConditionSpecific}}, an information structure $\mu$ consists of private information structures $\mu^1$ and $\mu^2$ defined with $Q^i, i=1,2$. Define $\mu^i$ as the joint probability measure induced on $\mathcal{P}(\mathbb{X} \times \mathbb{Y}^i)$ by measurement channel $Q^i$ with input distribution $\sy{\zeta}(dx)$. \sy{For our analysis, we will allow policies to be randomized with independent randomness. Which is to say, the set of all admissible measurable policies $\Gamma^i$ will be the set of all measurable functions $\gamma^i$, where $u^i = \gamma^i(y^i, \omega^i)$ for some independent noise variable $\omega^i$. Admissible randomized policies are stochastic kernels from $\mathbb{Y}^i$ to $\mathbb{U}^i$.} Under conditional independence, let us define the following cost functional for a single-stage setup:
\begin{eqnarray*}
&&J(\sy{\zeta},\mu^1,\gamma^1,\mu^2,\gamma^2)= E^{Q^1,Q^2,{\bf \underline{\gamma}}}_{\sy{\zeta}}\big[ c(x,u^1,u^2)\big]  \\
   && =\int_{\mathbb{X} \times \mathbb{Y}} c(x,\gamma^1(y^1),\gamma^2(y^2)) Q^1(dy^1|x)Q^2(dy^2|x) \sy{\zeta}(dx)\sy{.}
\end{eqnarray*}

Suppose that DM $1$ (the minimizer) wishes to minimize the cost and DM $2$ (the maximizer) wishes to maximize the cost. Let $\gamma^1$ and $\gamma^2$ be defined as earlier for each decision maker.  

\begin{definition} Given an information structure $\mu$, we say that $\gamma^{1,*}, \gamma^{2,*}$ is an equilibrium for the zero-sum game if
\begin{eqnarray*}\inf_{\gamma^1 \in \Gamma^1} J(\sy{\zeta},\mu^1,\gamma^{1},\mu^2,\gamma^{2,*}) &=& J(\sy{\zeta},\mu^1,\gamma^{1,*},\mu^2,\gamma^{2,*}) \\ &=& \sup_{\gamma^2 \in \Gamma^2}  J(\sy{\zeta},\mu^1,\gamma^{1,*},\mu^2,\gamma^2). \nonumber\end{eqnarray*}
\end{definition}
Let $V_{G}^{\mu}(\gamma^{1}, \gamma^{2})$ be the expected value of the cost function $c$ for the maximizer, for some game $G$, given information structure $\mu$ and strategies $(\gamma^{1}, \gamma^{2})$ for the minimizer and maximizer, respectively:
\begin{align*}
V^{\mu}_{G}(\gamma^{1}, \gamma^{2}) = \int c(x, \gamma^1(y^1), \gamma^2(y^2))Q^1(dy^1|x)Q^2(dy^2|x)\sy{\zeta}(dx)\sy{.}\nonumber
\end{align*}

Let $V^{*}(G, \mu)$ be $V_{G}^{\mu}(\gamma^{1}, \gamma^{2})$ where $(\gamma^{1}, \gamma^{2})$ are chosen to be the equilibrium strategies for the players. 

\begin{definition}
For fixed $\mathbb{X}, \mathbb{Y}^1, \mathbb{Y}^2$ and $\zeta$ such that $x \sim \zeta$, we say that an information structure $\mu$ is \textit{better for the maximizer} than information structure $\sy{\nu}$ (written as $\sy{\nu} \lesssim \mu$) over all games in a class of games $\mathbb{G}$ if and only if for all games $G$ in $\mathbb{G}$:
\begin{align*}
V^{*}(G, \mu) \geq V^{*}(G, \sy{\nu})\sy{.}\nonumber    
\end{align*}
\end{definition}

\begin{definition}\label{garbling2}
We denote by $\kappa^i \mu$ the information structure in which player $i$'s information from $\mu$ is \textit{garbled} by a stochastic kernel $\kappa^i$. \sy{We let {-i} denote the other player in the game.} Explicitly, this means the information structure becomes:
\begin{equation}
    (\kappa^{i}\mu)(B,dy^{-i}, dx)  = \int_{\mathbb{Y}^i}\kappa^{i}(B|y^i)\mu(dy^i,dy^{-i},dx), \: B \in \mathcal{B}(\mathbb{Y}^i)\sy{.} \nonumber
\end{equation}
\end{definition}

We use $K^i$ to denote the space of all such stochastic kernels $\kappa^i$ for player $i$. 

\begin{theorem} [P\k{e}ski \cite{Peski}]\label{PeskiTheorem} Let $\mathbb{X}, \mathbb{Y}^1, \mathbb{Y}^2$ be finite. For any two information structures $\mu$ and $\sy{\nu}$, $\mu$ is better for the maximizer than $\sy{\nu}$ over all games with finite action spaces $\mathbb{U}^1, \mathbb{U}^2$ if and only if there exist kernels $\kappa^i \in K^i, \sy{i = 1, 2}$, such that
\begin{align*}
\kappa^{\sy{1}} \sy{\nu} &= \kappa^{\sy{2}}\mu, 
\end{align*}
In particular, under Assumption \ref{infoStrucConditionSpecific}, we have the more explicit characterization with
\[\kappa^{\sy{1}}Q^1_\sy{\nu} = Q^1_\mu \quad \textnormal{and} \quad Q^2_\sy{\nu} = \kappa^{\sy{2}}Q^2_\mu.\] 
Where $Q^i_\mu$ and $Q^i_\sy{\nu}$ are the measurement channels for player $i$ under information structures $\mu$ and $\sy{\nu}$, respectively. 
\end{theorem}

In this paper we will obtain a standard Borel generalization of this result. 

\subsection{Team Theoretic Setup}

For completeness, we also discuss the team theoretic setup in our review. 

\sy{Lehrer, Rosenberg and Shmaya extended Blackwell's ordering of information structures to team problems in finite-space settings for various solution concepts, including Nash equilibrium and several forms of \textit{correlated} equilibrium, in \cite{LehrerRosenbergShmaya}. For these results to hold, various degrees of correlation between the players' private signals is allowed. These solution concepts for correlated equilibrium are adopted from \cite{Forges}, which builds on ideas first introduced in \cite{Aumann1987}. These provide an ordering of information structures for static stochastic team problems. Related results are discussed in \cite[Chapter 4]{YukselBasarBook}. Recently, advances have been made in understanding the topological properties of strategic measures in team problems in \cite{YukselSaldiSICON17}.} 


\subsection{Contributions}

In this paper, we will derive a standard Borel counterpart of Theorem \ref{PeskiTheorem}. While obtaining our results, we will also derive conditions for the existence of saddle points in Bayesian zero-sum games in standard Borel spaces, as well as a converse theorem to Blackwell's ordering of information structures in the infinite setup. 

The contributions of this paper are as follows:

\begin{itemize}

\item[(i)] We will derive a standard Borel counterpart of Theorem \ref{PeskiTheorem} characterizing an ordering of information structures for zero-sum games (Theorem \ref{Infinite}).

\item[(ii)] We present two supporting results: (a) As a minor technical contribution, we present sufficient conditions for the existence of saddle points in Bayesian zero-sum games with incomplete information in standard Borel spaces (Theorem \ref{Equil}). This will build on placing an appropriate topology on the space of policies adopted by the decision makers. Our analysis generalizes existing results in the literature, notably \cite{milgrom1985distributional} and \cite{mertens2015repeated}, though as we note in the paper our generalization is rather technical and the conditions in \cite{milgrom1985distributional} are nearly equivalent to ours. (b) As a further supporting theorem, we will present a partial converse to Blackwell's ordering theorem for standard Borel spaces, using a separating hyperplane argument and properties of locally convex spaces (Theorem \ref{ConverseBlackwellOrig}). This presents an explicit, self-sufficient derivation for a converse theorem to be utilized in our main theorem, though related comprehensive results have been reported in the literature, as we note in the paper.

\end{itemize}

\section{Supporting Results on Existence of Saddle-Points and Comparison for Zero-Sum Games with Standard Borel Spaces}

\subsection{On Existence of Saddle-Points and Equilibria}

Prior to focusing in on the ordering of information structures, we present a supporting result regarding when equilibrium solutions to zero-sum games exist. In the finite case, equilibrium solutions always exist \cite{VonNeumann} (through e.g. \cite[Theorem 4.4]{basols99}), but this does not hold true in general \cite{IntegerGame}. Theorem \ref{Equil} below gives sufficient conditions for equilibrium solutions to exist for games with incomplete information. 

The existence of a value for games with incomplete information has been studied rather extensively. For readers' convenience, and as a direct proof, we present the result below; to our knowledge our statement and conditions have not been stated in the prior literature, though results nearly equivalent to ours have been noted rather indirectly: Most notably, \sy{Milgrom and Weber} present an existence result for more general games in \sy{\cite[Theorem 1]{milgrom1985distributional}, which} presents conditions whose generality is difficult to interpret: a careful look at condition R1 in \cite[p. 625]{milgrom1985distributional} leads to the conclusion that the authors have nearly (but not exactly) the same condition (ii) we note below; that is continuity of the cost function in the actions for every fixed hidden state variable $x$ is sufficient, though the statements given in \cite{milgrom1985distributional} imposes conditions that are not conclusive on this; we attribute this to the fact that the authors utilize \cite[Prop. 1(c)]{milgrom1985distributional} without establishing its relation to item (ii) below (due to the measurability requirement in the statement of \cite[Prop. 1(c)]{milgrom1985distributional}). Our analysis affords the simplicity and generality in the condition, since we build on the $w$-$s$ topology, rather than weak topology and directly Lusin's theorem \cite{Dud02} as followed in \cite{milgrom1985distributional} (we also note that the relation between weak and $w$-$s$ topologies on probabilities defined on product spaces with a fixed marginal can in fact be established using Lusin's theorem). Hence, in a strict sense, our conditions are more direct and general as stated.

The comprehensive book \cite[Proposition III.4.2.]{mertens2015repeated} imposes continuity in all the variables (unlike what is presented below). Furthermore,  \cite[Proposition III.4.2.]{mertens2015repeated} builds on a topology construction on policies which is different from what we present here; regarding the construction in \cite{mertens2015repeated} we would like to caution that in the absence of absolute continuity conditions on the information structure, this construction may lead to a lack of closedness on the sets of admissible policies (or {\it strategic measures}) as the counterexample \cite[Theorem 2.7]{YukselSaldiSICON17} reveals: in this counterexample, which would reduce to the setup studied here with $y^1=y^2=y$, a sequence of policies is constructed so that for each element of the sequence\sy{,} the action variables of the two decision makers are conditionally independent given their measurements, but the setwise (and hence, weak) limit of the sequence is not conditionally, or otherwise, independent; and thus the limit measure does not belong to the original information structure. For a more detailed discussion, we refer the reader to \cite[Section 7.2]{saldiyukselGeoInfoStructure}.



\begin{theorem}[Existence of Equilibria]\label{Equil}
For a given game, assume that Assumption \ref{infoStrucCondition} holds. Further, let the following hold.
\begin{enumerate}[nosep]
    \item[(i)] The action spaces of players, $\mathbb{U}^1, \mathbb{U}^2$, are compact.
    \item[(ii)] The cost function $c$ is bounded and continuous in players' actions, for every state of nature $x$.
\end{enumerate}    
Then an equilibrium exists under possibly randomized policies, and so there exists a value of the zero-sum game. 
\end{theorem}


\begin{proof}
\textit{\textbf{Step (1)}}: By Assumption \ref{infoStrucCondition}, we can reformulate the problem in a new probability space in which the measurements are independent from the unknown variable $x$. This reformulation, called an {\it independent-measurements reduction}, is essentially due to Witsenhausen \cite{WitsenhausenIan}, with a detailed discussion in \cite[Section 2.2.]{YukselWitsenStandardArXiv}, see Figure \ref{L1}.


\begin{figure}[h]
\centering
\epsfig{figure=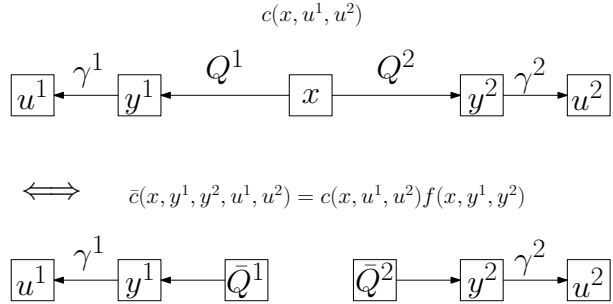,height=4cm,width=8cm}
\caption{Reformulation of two information structures (with respect to an independent measurements reduction) \label{L1}}
\end{figure}

%

The main benefit of this approach is to define a compact/convex policy space for the players (e.g. see \cite[Section 2.2]{YukselSaldiSICON17}). To complete this reformulation, we note the following holds for some function $f$ and reference probability measures $\bar{Q}^i$:
\begin{align*}
    &P(dx, dy^1, dy^2, du^1, du^2) = \zeta(dx)f(x,y^1,y^2)\bar{Q}^1(dy^1)1_{\{\gamma^1(y^1)\in du^1\}}\bar{Q}^2(dy^2)1_{\{\gamma^2(y^2)\in du^2)\}}\sy{.} \nonumber
\end{align*}
where $1_{\{\cdot\}}$ is the indicator function. Thus, the value function for the game can be written as:
\begin{equation}
    V^{\mu}_{G}(\gamma^1, \gamma^2) = \int f(x,y^1,y^2)c(x, u^1, u^2)\bar{Q}^1(dy^1)\bar{Q}^2(dy^2)\zeta(dx)\sy{.} \nonumber
\end{equation}
We then create a new cost function $c(x,u^1,u^2,y^1,y^2) = c(x,u^1,u^2)f(x,y^1,y^2)$.

\textit{\textbf{Step (2)}}: 
Let $x \in \mathbb{X}$ be the random state of nature. Let $\gamma^1, \gamma^2$ be the policies for the players, and $u^1, u^2$ be the resulting actions chosen by the players. We allow for policies $\gamma^i$ where $u^i$ is chosen in a random way, i.e. $u^i = \gamma^i(y^i, \omega^i)$, where $\omega^i$ is some $[0,1]$-valued independent random variable (we note that any randomized policy, defined as a stochastic kernel from $\mathbb{Y}^i$ to $\mathbb{U}^i$, admits such a stochastic realization; see \cite[Lemma 1.2]{gihman2012controlled}, or \cite[Lemma 3.1]{BorkarRealization}).

\textit{\textbf{Step (3)}}: Let $c$ be the reformulated cost function of this game, under the new product probability measure, we have: 
%
\begin{align*}
    &V_{G}^{\mu}(\gamma^1, \gamma^2) = \int c(x, u^1, u^2, y^1, y^2) (\Bar{Q}^1\gamma^1)(dy^1, du^1)(\Bar{Q}^2\gamma^2)(dy^2, du^2) \zeta(dx)\sy{.}\nonumber
\end{align*}
Here, $(\Bar{Q}^1\gamma^1)(dy^1, du^1)$ and $(\Bar{Q}^2\gamma^2)(dy^2, du^2)$ are the probability measures induced on the measurement and the action variables. By independence due to the reduction, we can consider the expected cost as a function of the reduced-form policies: $V_{G}^{\mu}(\gamma^1, \gamma^2) = V_{G}^{\mu}(\Bar{Q}^1\gamma^1, \Bar{Q}^2\gamma^2)$. 
Now, without loss of generality, we fix $\Bar{Q}^1\gamma^1$, allowing us to express the above equation in the following form:
\begin{align*}
&V_{G}^{\mu}(\bar{Q}^1\gamma^1, \bar{Q}^2\gamma^2) = \int (\Bar{Q}^2\gamma^2)(dy^2, du^2)\int c(x, u^1, u^2, y^1, y^2)(\Bar{Q}^1 \gamma^1)(dy^1, du^1)\zeta(dx)\sy{.}\nonumber
\end{align*}

Let $\int c(x, u^1, u^2, y^1, y^2)(\Bar{Q}^1 \gamma^1)(dy^1, du^1)\zeta(dx)$ be defined as $\Bar{c}(u^2, y^2)$.

Now that we have an independent-measurements reduction, we will (similar to the analysis from \cite{milgrom1985distributional, BorkarRealization, YukselSaldiSICON17}), identify, almost surely, every admissible policy with a probability measure on the product space: we adopt the view that, given game $(G, \mu)$, $\bar{Q}^i\gamma^i$ is a probability measure on $\mathbb{Y}^i \times \mathbb{U}^i$ with fixed marginal $\bar{Q}^i(dy^i)$ on $\mathbb{Y}^i$. Let $\Gamma^i$ denote the space of all such measures since every $\bar{Q}^i\gamma^i$ can be identified with an element in $\Gamma^i$ almost surely. The pairing of an information structure and a policy induces a probability measure $P$ on the five-tuple: $(\mathbb{X}, \mathbb{Y}^1, \mathbb{Y}^2, \mathbb{U}^1, \mathbb{U}^2)$, with 

\begin{equation} P(dx, dy^1, dy^2, du^1, du^2) = \gamma^1(du^1|y^1)\gamma^2(du^2|y^2)Q^1(dy^1|x)Q^2(dy^2|x)\zeta(dx).\nonumber
\end{equation}

This construction allows us to obtain a proper topology to work with for spaces of policies with desirable convexity and compactness properties. 


We now recall the $w$-$s$ topology \cite{schal1975dynamic} on the set of probability measures ${\cal P}(\mathbb{X} \times \mathbb{U})$; this is the coarsest topology under which $\int f(x,u) \nu(dx,du): {\cal P}(\mathbb{X} \times \mathbb{U}) \to \mathbb{R}$ is continuous for every measurable and bounded $f$ which is continuous in $u$ for every $x$ (but unlike weak topology, $f$ does not need to be continuous in $x$). \sy{We note that functions which are continuous in one argument and measurable in the other are sometimes referred to as Carath\'eodory functions.} Now, since the exogenous variables are fixed, weak convergence in this setting is equivalent to $w$-$s$ convergence (see \cite{YukselWitsenStandardArXiv}), and continuity in the exogenous variable is not needed here. Consider a sequence of actions $(\bar{Q}^{2}\gamma^{2})_m (dy^{2}, du^{2})$ which converges to $(\bar{Q}^{2}\gamma^{2})(dy^{2}, du^2)$ weakly. We have that $\Bar{c}(u^2, y^2)$ is continuous in $u^2$. Since $\mu$ is fixed, the marginals on $\mathbb{Y}^2$ are fixed. Therefore, by \cite[Theorem 3.10]{schal1975dynamic} (or \cite[Theorem 2.5]{balder2001}), we can use the $w$-$s$ topology on the set of probability measures $\mathcal{P}(\mathbb{Y}^2 \times \mathbb{U}^2)$. And so we have continuity of $V^{\mu}_G(\bar{Q}^1\gamma^1, \bar{Q}^2\gamma^2)$ in $\bar{Q}^2\gamma^2$ in the $w$-$s$ topology and, by the equivalence in this setting, the weak topology.

This also holds for continuity in $\bar{Q}^1\gamma^1$ in the reverse case where we fix $\bar{Q}^2\gamma^2$. Therefore, in general, we have that $V_{G}^{\mu}(\cdot, \cdot)$ is continuous in $(\Bar{Q}^i \gamma^i)$ when $(\Bar{Q}^{-i} \gamma^{-i})$ is fixed.

\textit{\textbf{Step (4)}}: Let $\Gamma = \{\lambda \in \mathcal{P}(\mathbb{Y} \times \mathbb{U}): \lambda_{\mathbb{Y}} = \bar{Q}\}$ be our reduced policy space, where $\bar{Q}$ is the fixed marginal of the measure $\lambda$ on $\mathbb{Y}$. Following from \cite[Section 2.1]{YukselSaldiSICON17}, the space of all $\bar{Q}^i\gamma^i$ (which we denote by $\Gamma^i$) is compact under weak convergence. 

\textit{\textbf{Step (5)}}: We observe that $V_{G}^{\mu}(\Bar{Q}^1\gamma^1, \Bar{Q}^2\gamma^2)$ is linear and hence is both concave and convex in each entry. For completeness, we establish this linearity result. Take $\theta \in (0, 1)$. Then, without loss of generality, we fix $\Bar{Q}^1\gamma^1$ and obtain the following:
\begin{align*}
 & V_{G}^{\mu}(\Bar{Q}^1\gamma^1, \theta\Bar{Q}^2\gamma^2 + (1 - \theta)\Tilde{Q}^2\Tilde{\gamma}^2) \\ = &\int (\theta\Bar{Q}^2\gamma^2 + (1 - \theta)\Tilde{Q}^2\Tilde{\gamma}^2)(dy^2, du^2) \int_{y^2, u^2} \Bar{c}(x, u^2, y^2) \\
 = &\int (\theta\Bar{Q}^2\gamma^2)(dy^2, du^2)\int_{y^2, u^2} \Bar{c}(x, u^2, y^2) + \int(1 - \theta)(\Tilde{Q}^2\Tilde{\gamma}^2)(dy^2, du^2) \int_{y^2, u^2} \Bar{c}(x, u^2, y^2) \\
    = &\theta V_{G}^{\mu}(\Bar{Q}^1\gamma^1,\Bar{Q}^2\gamma^2) + (1 - \theta) V_{G}^{\mu}(\Bar{Q}^1\gamma^1, \Tilde{Q}^2\Tilde{\gamma}^2)\sy{.} \nonumber
\end{align*}

Lastly, we recall that, under the weak topology, the space of probability measures is a metric space, and thus our spaces $\Gamma^i$ are Hausdorff spaces. 

Since $V_{G}^{\mu}(\Bar{Q}^1\gamma^1, \Bar{Q}^2\gamma^2)$ is continuous, and convex/concave in the compact Hausdorff spaces $\Gamma^i$, we have the following equality \cite[Theorem 1]{Fan}:
\begin{align*}
    \min_{Q^1\gamma^1}\max_{Q^2\gamma^2} V_{G}^{\mu}(Q^1\gamma^1, Q^2\gamma^2) = \max_{Q^2\gamma^2}\min_{Q^1\gamma^1}V_{G}^{\mu}(Q^1\gamma^1, Q^2\gamma^2)\sy{.}\nonumber
\end{align*}

This establishes a (saddle-point) equilibrium for the game.
\end{proof}

Thus, we have obtained an existence result for the value of the games considered, and also provided an approach to topologize and convexify/compactify the policy spaces.

\sy{For completeness, as a standalone contribution to the literature, we also present the following theorem, which is a mild relaxation of the theorem above, removing the requirements for the {\it independent-measurements} reduction. However we will work with the {\it independent-measurements} reduction for the rest of the paper, and so the theorem above is sufficient for the main results of this paper. The proof follows similarly to Theorem \ref{Equil}}. 

\sy{\begin{theorem}[Existence of Equilibria with a Further Relaxation]\label{Equil2}
For a given game, assume the following hold.
\begin{enumerate}[nosep]
    \item[(i)] The action spaces of players, $\mathbb{U}^1, \mathbb{U}^2$, are compact.
    \item[(ii)] The cost function $c$ is bounded and continuous in players' actions, for every state of nature $x$.
\end{enumerate}    
Then an equilibrium exists under possibly randomized policies, and so there exists a value of the zero-sum game. 
\end{theorem}

\begin{proof}[Sketch.]
Here, we let $\mu_{\mathbb{Y}^i}$ denote the marginal of the information structure $\mu$ on $\mathbb{Y}^i$. We will combine our policies with these marginals to form the product measures $\mu_{\mathbb{Y}^1}(dy^1)\gamma^1(du^1|y^1)$ and $\mu_{\mathbb{Y}^2}(dy^2)\gamma^2(du^2|y^2)$ on the players' measurement and action spaces. We will denote these measures by $(\mu_{\mathbb{Y}^1}\gamma^1)$ and $(\mu_{\mathbb{Y}^2}\gamma^2)$

Similar to the previous proof, without loss of generality, we fix Player 2's strategy $\gamma^2$. Then we have:
\begin{equation}
  V_{G}^{\mu}(\gamma^1, \gamma^2) = \int \mu_{\mathbb{Y}^1}(dy^1)\gamma^1(du^1|y^1)\left(\int_{\mathbb{X} \times \mathbb{Y}^2 \times \mathbb{U}^2} \zeta(dx)\mu(dy^2 | dy^1) \gamma^2(du^2|y^2) c(x,u^1,u^2)\right) \nonumber
\end{equation}

We observe that we can equivalently write $V^{\mu}_G$ as a function of the product measures $(\mu_{\mathbb{Y}^1}\gamma^1)$ and $(\mu_{\mathbb{Y}^2}\gamma^2)$, since the marginals of $\mu$ are fixed. 

Let $\int_{\mathbb{X} \times \mathbb{Y}^2 \times \mathbb{U}^2} \zeta(dx)\mu(dy^2 | dy^1) \gamma^2(du^2|y^2) c(x,u^1,u^2)$ be defined as $\Bar{c}(u^1,y^1)$.

We can observe that, by assumption, $\bar{c}(u^1, y^1)$ is bounded, and is continuous in $u^1$. Furthermore, it is also evident that $\bar{c}(u^1, y^1)$ is measurable in $y^1$. By the same arguments of \textbf{Step (2)} of the preceding theorem, via the machinery of the $w$-$s$ topology \cite{schal1975dynamic},
 we can show that $\bar{c}(u^1, y^1)$ is continuous in $(\mu_{\mathbb{Y}^1}\gamma^1)$ under $w$-$s$ convergence, and thus also under weak convergence. 
This also holds for continuity in $(\mu_{\mathbb{Y}^2}\gamma^2)$ in the reverse case where we fix $\gamma^1$. 


Following from \cite[Section 2.1]{YukselSaldiSICON17}, the space of all $\mu_{\mathbb{Y}^i}^i\gamma^i$ is compact under weak convergence, and we can observe that $V_{G}^{\mu}(\mu_{\mathbb{Y}^1}\gamma^1, \mu_{\mathbb{Y}^2}\gamma^2)$ is linear and hence is both concave and convex in each entry.

The existence of a (saddle-point) equilibrium for the game then follows by \cite[Theorem 1]{Fan}. 
\end{proof}
}

\subsection{On a Partial Converse to Blackwell Ordering in the Standard Borel Setup}

In addition to requiring conditions for the existence of equilibrium solutions in the infinite case, we need to address the extension of Blackwell's ordering of information structures to the infinite case, as this will form a key aspect of the proof of the main result of this paper, Theorem \ref{Infinite}. 

Here, we present a partial converse to Blackwell's theorem. 

The forward direction to Blackwell's theorem holds in the infinite case (see \cite[Theorem 4.3.2]{YukselBasarBook}), i.e. when $\mathbb{X}$, $\mathbb{Y}$ are standard Borel spaces for a single-player setup, $\sy{\nu}$ being a stochastically degraded version of $\mu$ implies that $\mu$ is more informative than $\sy{\nu}$ over all single-player decision problems with standard Borel action spaces and bounded cost functions that are continuous in the player's action for every state of nature. 

As noted earlier, related results were presented by C. Boll in 1955 in an unpublished thesis paper \cite{CBoll}. Le Cam presents a summary of these results in \cite{LeCamReview}, with a detailed review reported in \cite{torgersen1991comparison}. The approach in the literature often builds on the construction of {\it dilatations} of conditional probability measures, which is related to Blackwell's comparison of experiments theorem through what is known as the Blackwell-Sherman-Stein theorem. A detailed comparative analysis is provided further below. Our main contribution here is an explicit converse compatible with the conditions on existence results presented in the previous section and a comparison to be presented in the next section. This result serves as a supporting step with a direct proof; the results reported in the literature are often very technical and the explicit implication for our setup is not evident {\it a priori} as we discuss in the next subsection.

\sy{We note that our setup differs slightly from that of Blackwell in \cite{Blackwell}, contributing to the fact that this is a \textit{partial} converse to Blackwell's result. In Blackwell's original setup with finite $\mathbb{X}$, information structures could be compared over different priors on $\mathbb{X}$ as the comparison would apply uniformly to all such prior measures that satisfy a positivity condition on each of the finitely many outcomes. In our setup, since the space is possibly uncountable, we consider a fixed prior measure on $\mathbb{X}$.} 


\begin{theorem}\label{ConverseBlackwellOrig}
Let us consider a single player whose goal is to minimize the value of the cost function \sy{$c$} for a set of single-player decision problems. We assume the measurement is absolutely continuous in the following sense: there exists a function $f$ and a reference probability measure $\bar{Q}$ such that for all Borel $S$:
    \begin{align*}
         P(y \in S | x)  = \int_S f(y, x)\bar{Q}(dy)\sy{.}\nonumber
    \end{align*}
 If $\mathbb{Y}$ is compact and an information structure $\mu$ is more informative than another information structure $\sy{\nu}$ over all single-player decision problems with compact standard Borel action spaces and bounded cost functions \sy{$c: \mathbb{X} \times \mathbb{U} \rightarrow \mathbb{R}$} that are continuous in $u$ for every $x$, then $\sy{\nu}$ must be a garbling of $\mu$ in the sense of Definition \ref{garbling}.
\end{theorem}

\begin{proof}
We note that under the conditions of the theorem, an optimal policy (which is also deterministic) exists for every information structure (see Theorem 3.1 in \cite{YukselOptimizationofChannels}).

\textit{\textbf{Step (1)}}: Let $\zeta$ be the fixed probability distribution on $\mathbb{X}$ for any given decision problem in our set. Take information structures $\mu, \sy{\nu} \in \{\mathcal{P}(\mathbb{X} \times \mathbb{Y}) : P_{\mathbb{X}} = \zeta\}$, where $\mu$ is more informative than $\sy{\nu}$ in Blackwell's sense (i.e. $J(\zeta, \mu, \gamma^*) \geq J(\zeta, \sy{\nu}, \gamma^*)$ over all games with bounded cost functions \sy{$c$} that are continuous in $u$). 

Take the space $K \mu$, a subset of $\mathcal{P}(\mathbb{X} \times \mathbb{Y})$, to be the space of all possible garblings of $\mu$, where the garblings are from $\mathbb{Y}$ to $\mathbb{Y}$.

\textit{\textbf{Step (2)}}: We now establish the weak compactness of the space of all garbled information structures. 
First, observe that the set of all induced garblings on the product space (involving all of $K$) inducing probability measures of the form
\[\sy{P_{K}}(dx,dy,d\tilde{y}) := \mu(dx,dy)K(d\tilde{y}|y)\sy{,}\]
 leads to a weakly pre-compact space in the space of probability measures on $\mathbb{X} \times \mathbb{Y} \times \mathbb{Y}$. If closedness can also be established, this would lead to a weakly compact space.
 This follows from the proof of \cite[Theorem 5.6]{YukselWitsenStandardArXiv}: since the marginals on $\mathbb{X} \times \mathbb{Y}$ are fixed, any limit of a weakly converging \sy{sequence} will also satisfy the property that the limit is a garbling of the original information structure. For completeness, we present the following: With $P_K(dx,dy,d\tilde{y}) = K(d\tilde{y}|y)\mu(dx,dy)$, consider a weakly converging sequence $P_{K_n}(dx,dy,d\tilde{y})$. We will show that the weak limit also admits such a garbled structure. Let $P_{K_n}(dx,dy,d\tilde{y})$ converge weakly to $P(dx,dy,d\tilde{y})$. Then, for every continuous and bounded $g$
 \[ \int g(x,y,\tilde{y})P_K(dx,dy,d\tilde{y}) =  \int \bigg( \int g(x,y,\tilde{y})\mu(dx | dy) \bigg) P_{K_n}(dy,d\tilde{y})\sy{.} \]
Since the marginal on $y$ is fixed, even though the function $\int g(x,y,\tilde{y})\mu(dx | dy)$ is only measurable and bounded in $y$ and is continuous in $\tilde{y}$, $w$-$s$ convergence is equivalent to the weak convergence of $P_{K_n}(dy,d\tilde{y})$ and as a result we have that
\[\int \bigg( \int g(x,y,\tilde{y})\mu(dx | dy) \bigg) P_{K_n}(dy,d\tilde{y}) \to \int \bigg( \int g(x,y,\tilde{y})\mu(dx | dy) \bigg) P(dy,d\tilde{y})\sy{.}  \]
As a result, $P$ decomposes as $P(dx,dy,d\tilde{y}) = \mu(dx,dy)\tilde{K}(d\tilde{y}|y)$ for some $\tilde{K}$. This establishes the weak compactness of the garbled information structure in the product space $\mathbb{X} \times \mathbb{Y} \times \mathbb{Y}$.

Now, take the projection of this space onto the measures on the first and the third coordinate; as a continuous image of a weakly compact set, this map will also be compact and gives us our space $K\mu$. 

Finally, $K\sy{\mu}$ is convex, since the space of stochastic kernels is convex. As a result, the space $K \mu$ of all possible garblings of $\mu$ is a convex and compact subset of ${\cal P}(\mathbb{X} \times \mathbb{Y})$ under the weak convergence topology. 

Now, assume there does not exist a stochastic kernel $\kappa \in K$ such that:
\begin{align*}
    \sy{\nu} = \kappa\mu \sy{.}\nonumber
\end{align*}
Which is to say, we assume $\sy{\nu}$ is not a garbling of $\mu$ and proceed with a proof by contradiction. Then, $K\mu \: \cap \: \sy{\nu} = \sy{\nu}$. That is, $\sy{\nu} \notin K \mu$. \\
\textit{\textbf{Step (3)}}: 
We now use the Hahn-Banach Separation Theorem for Locally Convex Spaces by treating the space of probability measures $\mathcal{P}(\mathbb{X} \times \mathbb{Y})$ as a locally convex space of measures (see \cite[Theorem 3.4]{rudin1991functional}). As such, since our spaces $K \mu$ and $\sy{\{\nu\}}$ are subsets of this space and are convex, closed and compact, in addition to being disjoint, we can separate them using a continuous linear map from $\mathcal{P}(\mathbb{X} \times \mathbb{Y}$) to $\mathbb{R}$. 

To apply \cite[Theorem 3.4]{rudin1991functional}, we require local convexity of $\mathcal{P}(\mathbb{X} \times \mathbb{Y})$, and so we define the locally convex space of probability measures with the following notion of convergence: We say that $\nu_n \rightarrow \nu$ if $\int f(x,y) \nu_n(dx,dy) \rightarrow \int f(x,y) \nu(dx,dy)$ for every measurable and bounded function which is continuous in $y$ for every $x$. We note that our measures must still have fixed marginal $\zeta$ on $\mathbb{X}$.

Since continuous and bounded functions {\it separate} probability measures (in the sense that, if the integrations of two measures with respect to continuous functions are equal, the measures must be equal), it follows from \cite[Theorem 3.10]{rudin1991functional} that we can represent every continuous linear map on $\mathcal{P}(\mathbb{X} \times \mathbb{Y})$ \sy{using the form $\int f(x,y) \sy{\nu} (dx, dy)$}
for some measurable and bounded function $f(x,y)$ continuous in $y$ for every $x$. It also follows from \cite[Theorem 3.10]{rudin1991functional} that, given this notion of convergence, $\mathcal{P}(\mathbb{X} \times \mathbb{Y})$ is a locally convex space. 

Therefore, we have the following statement from combining \cite[Theorem 3.4]{rudin1991functional} and \cite[Theorem 3.10]{rudin1991functional}: there exists a measurable and bounded function (continuous in $y$) $f: \mathbb{X} \times \mathbb{Y} \rightarrow \mathbb{R}$ and constants $D_1, D_2 \in \mathbb{R}$ where $D_1 < D_2$ such that:
\begin{align*}
 \langle\sy{\nu}, f\rangle \leq D_1, \: \langle \kappa\mu, f\rangle \geq D_2, \: \: \forall \: \kappa \in K\sy{.} \nonumber
\end{align*}
Where we use the following notation:
\begin{align*}
    \langle \sy{\nu}, f \rangle  = \int_{\mathbb{X} \times \mathbb{Y}} f(x,y) \sy{\nu}(dx, dy) \sy{.}\nonumber
\end{align*}

This gives us the following inequality:
$\langle \sy{\nu}, f \rangle < \langle \kappa\mu, f \rangle, \:  \forall \kappa \in K$. 

\textit{\textbf{Step (4)}}: Now consider the class of decision problems with bounded cost functions continuous in the actions, with compact $\mathbb{Y}$, $\mathbb{U}$, where $\mathbb{U} = \mathbb{Y}$. This is clearly a subset of all decision problems considered so far in the proof. Now let $f(x,y)$ be the separating function found above. Consider a game in this particular subclass where $f(x,y)$ is the cost function (which is valid since $\mathbb{U} = \mathbb{Y}$ and $f(x,y)$ is bounded continuous in $y$). We note that $\langle \sy{\nu}, f \rangle$ gives the expected value of the game with cost function $f(x,y)$ under information structure $\sy{\nu}$ when the player plays the identity policy $\gamma^{id}(y) = y$. \sy{We can observe the following:}
\begin{align}
\int_{\mathbb{X} \times \mathbb{Y}}f(x,y)\sy{\nu}(dx,dy) <& \int_{\mathbb{X} \times \mathbb{Y}}f(x,y)\kappa\mu(dx,dy),  \: \: \forall \kappa \in K\sy{,} \nonumber 
\end{align}
and hence,
\sy{\begin{align}
\int_{\mathbb{X} \times \mathbb{Y}}f(x,y)\sy{\nu}(dx,dy) < & \inf_{\kappa \in K} \int_{\mathbb{X} \times \mathbb{Y}}f(x,y)\kappa\mu(dx,dy) \nonumber \\
=& \inf_{\kappa \in K}\int_{\mathbb{X} \times \mathbb{Y}} f(x,y')\int_{\mathbb{Y}}\kappa(dy'|y)\mu(dx,y) \nonumber \\
=& \inf_{\kappa \in K}\int_{\mathbb{X} \times \mathbb{Y}} f(x, \kappa(\cdot|y))\mu(dx,dy) \sy{.}\nonumber 
\end{align}}
\sy{Where we define:
\[f(x, \kappa(\cdot|y)) := \int_{\mathbb{Y}}f(x,y')\kappa(dy'|y)\]}
\sy{Recalling that $\kappa(\cdot|y)$ has a functional representation $\gamma(y) = g(y, \omega)$ for some independent noise variable $\omega$, and since $K$ is the space of all stochastic kernels from $\mathbb{Y}$ to $\mathbb{Y}$, we can observe that this gives us:
\begin{equation}
 \inf_{\kappa \in K}\int_{\mathbb{X} \times \mathbb{Y}} f(x, \kappa(\cdot|y))\mu(dx,dy) = \inf_{\gamma \in \Gamma}\int_{\mathbb{X} \times \mathbb{Y}}f(x, \gamma(y))\mu(dx,dy). \nonumber
\end{equation}
Since we allow for randomized policies, this minimization is equivalent to finding the optimal policy $\gamma^* \in \Gamma$ for the cost function $f(x,y)$ under information structure $\mu$. And so we have:}
\begin{equation}
    \int_{\mathbb{X} \times \mathbb{Y}}f(x,y)\sy{\nu}(dx,dy) = J(\sy{\zeta}, \sy{\nu}, \gamma^{id}) < J(\sy{\zeta}, \mu, \gamma^*) = \inf_{\kappa \in K}\int_{\mathbb{X} \times \mathbb{Y}} f(x, \kappa(y))\mu(dx,dy)\sy{.} \nonumber
\end{equation}
Since we have found a game where, when playing its optimal policy, $\mu$ performs worse than $\sy{\nu}$ does under some policy, we have contradicted the fact that $\mu$ is better than $\sy{\nu}$. Therefore, there must exist a $\kappa \in K$ such that $\sy{\nu} = \kappa\mu$, and so $\sy{\nu}$ is a garbling of $\mu$. 
\end{proof}

This result will allow us to use both directions of Blackwell's ordering of information structures in the standard Borel-type setup we are considering for players in zero-sum games. \\

\noindent{\bf Dilatations as measures for comparisons of experiments and Strassen's theorem.} Strassen, in \cite[Theorem 2]{strassen1965existence}, presents a related result that is often \sy{invoked} when comparison of experiments is studied in infinite dimensional probability spaces, although the direct implication on Blackwell's ordering (in the sense needed in our main result to be presented in the next section) is not explicit as we note in the following. Likewise, \sy{Cartier, Fell, and Meyer} relate an ordering of information structures in terms of dilatations (where the hidden variable $x$ does not appear explicitly in the analysis) \sy{in \cite{cartier1964comparaison}}. A very concise yet informative review is in \cite[p. 130-131]{LeCamReview}. A detailed discussion on comparisons of information structures along the same approach is present in the comprehensive book \cite{torgersen1991comparison}. Both for completeness as well as to compare the findings, we present a discussion in the following.


Let $\Omega$ be a convex compact metrizable subset of a locally convex topological vector space. For Borel probability measures $\mu$ and $\nu$ write $\mu \prec \nu$ if and only if for all $y \in \mathcal{S} = \{\textnormal{all  continuous concave functionals on} \: \Omega\}$
\begin{align*}
    \int y \: d\mu \geq \int y \: d\nu. \nonumber
\end{align*}
\begin{theorem}\label{Strassen} \cite[Theorem 2]{strassen1965existence}
 $\mu \prec \nu$ if and only if there is a dilatation P such that $\nu = P\mu$, where a dilatation $P$ is a Markov kernel from $\Omega$ to $\Omega$ such that for all continuous affine functions $z$ on $\Omega$, $zP = z$.
 \end{theorem}

The condition $zP = z$ means that for any continuous affine function $z$ on $\Omega$:

\[\int_{\Omega}z(r)P(dr,\omega) = z(\omega), \quad \forall \omega \in \Omega.\]

Theorem \ref{Strassen} does not lead to a converse to Blackwell's theorem in the generality presented in Theorem \ref{ConverseBlackwellOrig}: Let $\Omega$ be the space of probability measures on $\mathbb{X}$. Let $\mu$ be an information structure that is more informative than another information structure $\sy{\nu} \:$ in Blackwell's sense. Let us restrict ourselves to decision problems where $\mathbb{U}$ is compact. Let $Q_\mu$ and $Q_{\sy{\nu}}$ be the measurement channels for the player under information structures $\mu$ and $\sy{\nu}$, respectively. By definition, we have for all measurable and bounded cost functions $c$ continuous in the actions:
\begin{align*}
\inf_{\gamma \in \Gamma}\int \sy{\zeta}(dx) Q_\mu(dy|x) c(x, \gamma(y)) \leq \inf_{\eta \in \Gamma}\int \sy{\zeta}(dx)Q_\sy{\nu}(dy|x)c(x, \eta(y))\sy{.} \nonumber
\end{align*}
\sy{Let $P^{\mu}(dy)Q(dx|y)$ be the alternative disintegration of the information structure $\mu$ following Bayes' rule. Likewise, perform the same disintegration for $\nu$. Then we can rewrite the above equation as (due to the measurable selection conditions as in the proof of Theorem 3.1 in \cite{YukselOptimizationofChannels}):
\begin{equation}\label{StrassenEqn}
  \int P^{\mu}(dy)(\inf_{u \in \mathbb{U}} \int Q(dx|y)c(x,u))  \leq \int P^{\sy{\nu}}(dy)(\inf_{u \in \mathbb{U}} \int Q(dx|y)c(x,u)) \sy{,}
\end{equation}
Now, we define:
\[\Pi^\mu(A) := \int_{\mathbb{Y}}P^{\mu}(dy)1_{Q^{\mu}(\cdot|A)}\]
We note that $\Pi^\mu$ is a probability measure on $\Omega$. Define $\Pi^{\nu}$ similarly. Then (\ref{StrassenEqn}) becomes
\[ \int \Pi^{\mu}(d\pi)(\inf_{u \in \mathbb{U}} \int \pi(dx) c(x,u))  \leq \int \Pi^{\nu}(d\pi)(\inf_{u \in \mathbb{U}}\int \pi(dx) c(x,u)),\]
with the interpretation that $\pi(dx)=Q(dx|y)$. Let $W^*(\pi) = \inf_{u \in \mathbb{U}} \int \pi(dx)c(x,u)$. Then we can rewrite this once again as:
\begin{align*}
    \int \Pi^{\mu}(d\pi)W^*(\pi) \leq \int \Pi^{\sy{\nu}}(d\pi)W^*(\pi)\sy{.} \nonumber
\end{align*}}
Since \sy{$\Pi^{\mu}$ and $\Pi^{\nu}$} give probability distributions on $\Omega$, and $W^*$ is a function over $\Omega$, we will have $\sy{\nu} \prec \mu$ in Strassen's sense if the above inequality holds for all continuous and concave functions over $\Omega$.

We can show that $W^*$ is continuous and concave in $\pi$ provided that additionally $c$ is continuous {\it both} in $x$ and $u$: Let $\pi_n \rightarrow \pi$ weakly. Let $u^{*}_n$ be optimal for $\pi_n$. Then:
\begin{align*}
    &|\int c(x, u^{*}_n) \pi_n(dx) - \int c(x, u^*)\pi(dx)| \\
    &\leq \max(\int c(x, u^{*}_n)(\pi_n(dx) - \pi(dx)), \int c(x, u^*)\pi_n(dx) - \pi(dx))\sy{.}
\end{align*}
We note that $\int c(x, u^{*}_n)(\pi_n(dx) - \pi(dx))$ goes to $0$ following \cite[Theorem 3.5]{Serfozo} or \cite[Theorem 3.5]{Lan81} (since the action space is compact, there always is a converging subsequence $u^*_{n_m} \to \bar{u}$ for some $\bar{u}$, and since for $x_n \to x$ we have that $c(x_{n_m},u^*_{n_m}) \to c(x,\bar{u})$ the result follows from a generalized convergence theorem under weak convergence). The second term converges to zero by the weak convergence of $\pi_n$ to $\pi$. We emphasize the requirement that $c$ is continuous in both $x$ and $u$, in Theorem \ref{ConverseBlackwellOrig} only continuity in $u$ was required (one can construct a simple counterexample, even when $\mathbb{U}$ is a singleton to show that continuity in $x$ is necessary for this argument to hold). Concavity of $W^*$ in the conditional measure $\pi(dx)$ follows from Proposition \ref{useful}. 

Now, if one can show that by using all bounded continuous cost functions $c$ and compact action spaces $\mathbb{U}$, the space of all continuous and concave functions on $\Omega$ is spanned by the space of all $W^*$ functions, then a converse can be attained through Strassen's result. We note here that every concave and upper semi-continuous $W$ can be written as an infimum of a family of affine functions (Fenchel-Moreau Theorem \cite{rockafellar1970convex}) and an analysis can be pursued towards this direction at least for the case where $c$ can be assumed to be continuous in both variables and the condition on $W$ is to be relaxed in Strassen's theorem. However, due to the conditions of upper semi-continuity of $W^*$ and the joint continuity of \sy{$c$} noted earlier in both the state and actions, the applicability of Strassen's theorem to our setup does not hold in the generality reported. 

In summary, our paper presents a general condition and a direct proof, while we recognize that Strassen's theorem and accordingly its proof could be further modified to allow for additional relaxations for arriving at a similar result. 

\section{Comparison of Information Structures for Zero-Sum standard Borel Bayesian Games}
We are now prepared to order information structures in the spirit of Theorem \ref{PeskiTheorem} for this standard Borel setup. We note that the following lemmas, theorem, and corollary also hold in the general finite case studied by P\k{e}ski, as they rely solely on the existence of equilibria (which are guaranteed to exist in the finite setup by von Neumann's min-max theorem, see \cite{VonNeumann}) and Blackwell's ordering of information structures. Therefore, these results also serve as a strict generalization of Theorem \ref{PeskiTheorem} to standard Borel Bayesian Games. We note here that the required absolute continuity conditions always hold for finite or countable spaces (in that one can always find a reference measure with respect to which all probability measures on a countable space is absolutely continuous).

\begin{definition}
 For fixed $\mathbb{X}$ with $x \sim \zeta$, and fixed $\mathbb{Y}^1, \mathbb{Y}^2$, we define a class of games $\Tilde{\mathbb{G}}_\zeta(\mathbb{X}, \mathbb{Y}^1, \mathbb{Y}^2)$ to be all games for which the players have compact action spaces and the cost function is bounded and continuous in players' actions for every state $x$. 
\end{definition}

\begin{lemma}\label{ForwardLemma}
Given fixed $\mathbb{X}$, $\zeta$, $\mathbb{Y}^1$, and $\mathbb{Y}^2$, for any information structure $\mu$ which satisfies Assumption 2.1 and any kernels $\kappa^i \in K^i$:
\begin{align*}
    \kappa^{\sy{2}}\mu \lesssim \mu \: \: \: \: \: and \: \: \: \: \: \mu \lesssim \kappa^{\sy{1}}\mu\sy{,}\nonumber
\end{align*}
over all games in $\Tilde{\mathbb{G}}_\zeta(\mathbb{X}, \mathbb{Y}^1, \mathbb{Y}^2)$. 
\end{lemma}

\begin{proof}
Let us consider the first relation.

Take an arbitrary zero-sum game $G \in \Tilde{\mathbb{G}}_\zeta(\mathbb{X}, \mathbb{Y}^1, \mathbb{Y}^2)$ with cost function $c$ and action spaces $\mathbb{U}^1$ and $\mathbb{U}^2$. Let $(\gamma^{1}, \gamma^{2})$ be the Bayesian Nash equilibrium policies for the players under information structure $\kappa^{\sy{2}}\mu$ and $(\eta^{1}, \eta^{2})$ be the Bayesian Nash equilibrium policies under information structure $\mu$. By our assumption on \sy{$\Tilde{\mathbb{G}}_\zeta(\mathbb{X}, \mathbb{Y}^1, \mathbb{Y}^2)$}, these policies exist [Theorem \ref{Equil}]. Let $Q^{i}$ be the measurement channel for player $i$ under information structure $\mu$.

The expected value of the cost for the maximizer under the first information structure is:

\begin{align*}
V_{G}^{\kappa^{\sy{2}}\mu}(\gamma^1, \gamma^2) = \int_{\mathbb{X} \times \mathbb{Y}^1 \times \mathbb{Y}^2} c(x, \gamma^1(y^1), \gamma^2(y^2))\sy{\kappa^2\mu(dx,dy^1,dy^2})\sy{.}\nonumber
\end{align*}

By definition, the equilibrium solution $(\gamma^1, \gamma^2)$ for $G$ under $\kappa^{\sy{2}}\mu$ is given by the solution to the min-max problem:
\begin{align*}
\min_{\theta^1 \in \Gamma^1}\max_{\theta^2 \in \Gamma^2} V_{G}^{\kappa^{\sy{2}}\mu}(\theta^1, \theta^2)\sy{.}\nonumber
\end{align*}

Therefore, since $\gamma^1$ is the minimizing policy under $\kappa^{\sy{2}}\mu$, by perturbing the minimizer's policy $\gamma^1$ to be the policy $\eta^1 \in \Gamma^1$ we have the following inequality (i.e. we make the minimizer no longer play her optimal policy):
\begin{align*}
V_{G}^{\kappa^{\sy{2}}\mu}(\gamma^1, \gamma^2) \leq V_{G}^{\kappa^{\sy{2}}\mu}(\eta^1, \gamma^2)\sy{.}\nonumber
\end{align*}

We now wish to compare the two quantities $V_{G}^{\kappa^{\sy{2}}\mu}(\eta^1, \gamma^2)$ and $V_{G}^{\mu}(\eta^1, \eta^2)$. To do so, fix $\eta^1$ across both terms and consider a cost function $\Tilde{c}(x, \theta^2(y^2)): \mathbb{X} \times \mathbb{U}^2 \rightarrow \mathbb{R}$ such that $\Tilde{c}(x, \theta^2(y^2)) = c(x, \eta^1(y^1)), \theta^2(y^2)) \: \forall \: \theta^2 \in \Gamma^2$. I.e., by holding the minimizer's strategy constant as $\eta^1$, we reduce $c$ to $\Tilde{c}$ such that we now have a cost function that only reflects dependence on the maximizer's policy when the minimizer's policy is held at $\eta^1$. Such a function $\Tilde{c}$ clearly exists, as the value of $\eta^1(y^1)$ is only dependent on $x$ (potentially in some stochastic way, in that it depends on $Q^1(y|x)$), when $\eta^1$ (and $\mu^1$) are constant, and so can be absorbed into the dependency of $\Tilde{c}$ on $x$. 

We can now compare the single-player decision problem for the maximizer given by cost function $\Tilde{c}$ and information structures $(\kappa^{\sy{2}}\mu)^{\sy{2}}$ and $\mu^{\sy{2}}$ (which we use to denote the maximizer's private information structures present in $\kappa^{\sy{2}}\mu$ and $\mu$, respectively\sy{, i.e. the marginals on $(\mathbb{X} \times \mathbb{Y}^2$)}). This is a single-player decision problem and as such can be treated using the forward direction to Blackwell's ordering of information structures \cite[Theorem 2]{Blackwell}, which holds in this infinite-dimensional case \cite{YukselBasarBook}. Since $\Tilde{c}$ and $c$ are equal over all strategies in $\Gamma^2$, we know that $\gamma^2$ and $\eta^2$ are still optimal policies for the maximizer to play under \sy{the} respective information structures for this game. Thus, since $(\kappa^{\sy{2}}\mu)^{\sy{2}}$ is a garbling of $\mu^{\sy{2}}$ by channel $\kappa^{\sy{2}}$, and since $\Tilde{c}(x, \eta^2(y^2)) = c(x, \eta^1(\mu^1(x)), \eta^2(y^2))$, we have that:
\begin{align*}
& V_{G}^{\kappa^{\sy{2}}\mu}(\eta^1, \gamma^2) \\ = &\int_{\mathbb{X} \times \mathbb{Y}^2} \Tilde{c}(x, \gamma^2(y^2))\sy{\kappa^{\sy{2}}\mu(dx, dy^2)} \\ \leq &\int_{\mathbb{X} \times \mathbb{Y}^2} \Tilde{c}(x, \eta^2(y^2))\sy{\mu(dx, dy^2)}\\ =  &V_{G}^{\mu}(\eta^1, \eta^2) \sy{.}\nonumber
\end{align*}

Putting this all together, we have $V_{G}^{\kappa^{\sy{2}}\mu}(\gamma^1, \gamma^2) \leq V_{G}^{\kappa^{\sy{2}}\mu}(\eta^{1}, \gamma^{2}) \leq V_{G}^{\mu}(\eta^1, \eta^2)$. Since this is true for any arbitrary game $G \in \Tilde{\mathbb{G}}_\zeta(\mathbb{X}, \mathbb{Y}^1, \mathbb{Y}^2)$, we have that $\kappa^{\sy{2}}\mu \lesssim \mu$.

A nearly identical argument can be applied to show that $\mu \lesssim \kappa^{\sy{1}}\mu$.
\end{proof}

Using a similar reasoning, we also develop the following converse result:

\begin{lemma}\label{ConverseLemma}
\textit{Take fixed $\mathbb{X}, \zeta$, fixed and compact $\mathbb{Y}^1, \mathbb{Y}^2$, and information structures $\sy{\nu}$ and $\mu$ which both satisfy Assumption 2.1. If $\sy{\nu} \lesssim \mu$ over all games in $\Tilde{\mathbb{G}}_\zeta(\mathbb{X}, \mathbb{Y}^1, \mathbb{Y}^2)$, then there exist kernels $\kappa^i \in K^i$ such that:}
\begin{align*}
    \kappa^{\sy{1}}\sy{\nu} = \kappa^{\sy{2}}\mu\sy{.}\nonumber
\end{align*}
In particular, under Assumption \ref{infoStrucConditionSpecific}, we have the more explicit characterization with
\[\kappa^{\sy{1}}Q^1_\sy{\nu} = Q^1_\mu \quad \textnormal{and} \quad Q^2_\sy{\nu} = \kappa^{\sy{2}}Q^2_\mu.\] 
Where $Q^i_\mu$ and $Q^i_\sy{\nu}$ are the measurement channels for player $i$ under information structures $\mu$ and $\sy{\nu}$, respectively. 
\end{lemma}

\begin{proof}
Let $(\gamma^1, \gamma^2)$ be the equilibrium solution under $\sy{\nu}$ and let $(\eta^1, \eta^2)$ be the equilibrium solution under $\mu$. Let $Q^i_{\sy{\nu}}$ and $Q^i_{\mu}$ be the measurement channel for player $i$ under the information structures $\sy{\nu}$ and $\mu$, respectively. As in Lemma \ref{ForwardLemma}, these equilibria exist and are the solutions of the standard min-max problem. 

Therefore, we have the following inequality:
\begin{align*}
V_{G}^{\sy{\nu}}(\gamma^1, \eta^2) \leq \max_{\theta^2 \in \Gamma^2}V_{G}^{\sy{\nu}}(\gamma^1, \theta^2)= V_{G}^{\sy{\nu}}(\gamma^1, \gamma^2)\sy{.}\nonumber
\end{align*}

Likewise, we can determine the following:
\begin{align*}
V_{G}^{\mu}(\eta^1, \eta^2) = \min_{\alpha^1 \in \Gamma^1}V_{G}^{\sy{\nu}}(\alpha^1, \eta^2)\leq V_{G}^{\mu}(\gamma^1, \eta^2)\sy{.}\nonumber
\end{align*}

In addition, by assumption that $\sy{\nu} \lesssim \mu$ we have that:
\begin{align*}
V_{G}^{\sy{\nu}}(\gamma^1, \gamma^2) \leq V_{G}^{\mu}(\eta^1, \eta^2)\sy{.}\nonumber
\end{align*}

From above, one observes that $V_{G}^{\sy{\nu}}(\gamma^1, \eta^2) \leq V_{G}^{\mu}(\eta^1, \eta^2)$. In the same manner as in Lemma \ref{ForwardLemma}, we hold $\eta^2$ constant across both terms and develop a reduced single-player cost function $\Tilde{c}$. \sy{Once again, we use $\nu^i$ and $\mu^i$ to denote the private (i.e. marginal) information structure for Player $i$ under $\nu$ and $\mu$, respectively}. We then have a single-player decision problem where we observe that $\gamma^1$ and $\eta^1$ are still the optimal policies for the minimizer for each respective information structure:
\begin{align*}
 J(\zeta,\sy{\nu}^1,\gamma^1) \nonumber &= \int_{\mathbb{X} \times \mathbb{Y}^1} \Tilde{c}(x, \gamma^1(y^1))\sy{\nu(dx, dy^1)} \nonumber \\ &\leq \int_{\mathbb{X} \times \mathbb{Y}^1} \Tilde{c}(x, \eta^1(y^1))\sy{\mu(dx,dy^1)} = J(\zeta, \mu^1,\eta^1) \sy{.}
\end{align*}

Since the inequality $V_{G}^{\sy{\nu}}(\gamma^1, \eta^2) \leq V_{G}^{\mu}(\eta^1, \eta^2)$ holds true for every arbitrary zero-sum game $G \in\Tilde{\mathbb{G}}_\zeta(\mathbb{X}, \mathbb{Y}^1, \mathbb{Y}^2)$, it holds for every game in the subclass $\hat{\mathbb{G}}$, defined here to be all games in $\Tilde{\mathbb{G}}_\zeta(\mathbb{X}, \mathbb{Y}^1, \mathbb{Y}^2)$ where the action space of the maximizer is $\mathbb{U}^2 = \{0\}$. Moreover, we observe that for any arbitrary bounded single-player cost function that is continuous in the player's action $\bar{c}(x, u^1): \mathbb{X} \times \mathbb{U} \rightarrow \mathbb{R}$\sy{,} there exists a two-player cost function $\hat{c}(x, u^1, u^2)$ corresponding to some game in $\hat{\mathbb{G}}$ such that $\hat{c}(x, u^1, u^2) = \bar{c}(x, u^1) \: \forall \: u^1 \in \mathbb{U}^1$ (following naturally from the fact that $u^2 = 0$ for these games). One such construction of $\hat{c}$ would be $\hat{c}(x, u^1, u^2) = \bar{c}(x, u^1) + u^2$; when played in $\hat{\mathbb{G}}$, $\hat{c}(x, u^1, u^2) = \bar{c}(x, u^1) \: \forall \: u^1 \in \Gamma^1$. We note that since $\bar{c}(x, u^1)$ is continuous in $u^1$ for all $x$ and is bounded, it is a valid for a game in $\hat{\mathbb{G}} \subset \Tilde{\mathbb{G}}_\zeta(\mathbb{X}, \mathbb{Y}^1, \mathbb{Y}^2)$ to use $\hat{c}(x, u^1, u^2)$. 

Therefore, $\bar{c}$ is a valid single-player reduction of $\hat{c}$ when the maximizer's strategy is held constant. Since this process can be done for any single-player cost function $\hat{c}$, we observe that for constant $\mathbb{U}^i$ and $\mathbb{X}$, the reduction of all measurable and bounded two-player cost functions that are continuous in player's actions with one player playing a constant strategy is surjective on the entire space of bounded single-player cost functions that are continuous in the player's action. Therefore, the inequality above will hold over all single-player cost functions that are bounded and continuous in action, since it holds over all arbitrary two-player game\sy{s} $G \in \Tilde{\mathbb{G}}_\zeta(\mathbb{X}, \mathbb{Y}^1, \mathbb{Y}^2)$. 

Lastly, we observe that since $c$ is framed such that a higher quantity is better for the maximizer, the minimizer wants to minimize the value of $\Tilde{c}$. Therefore, from the minimizer's perspective, the inequality above indicates that she can never perform worse under $\sy{\nu}$ than under $\mu$ over all single-player problems with bounded cost functions that are continuous in the player's action. Thus, by the converse direction to Blackwell's ordering of information structures \cite[Theorem 6]{Blackwell}, which by the lemma assumptions and the restrictions on the class $\Tilde{\mathbb{G}}_\zeta(\mathbb{X}, \mathbb{Y}^1, \mathbb{Y}^2)$ (namely compactness of $\mathbb{Y}^i$ and $\mathbb{U}^i$) holds in this infinite setup due to Theorem \ref{ConverseBlackwellOrig}, we have that $\mu^{\sy{1}}$ must be a garbling of $\sy{\nu}^{\sy{1}}$.

In a similar manner, by observing that $V_{G}^{\sy{\nu}}(\gamma^1, \gamma^2) \leq  V_{G}^{\mu}(\gamma^1, \eta^2)$ one discovers that $\sy{\nu}^{\sy{2}}$ must be a garbling of $\mu^{\sy{2}}$.

Therefore, we have that the minimizer's channel in $\sy{\nu}$ is garbled in $\mu$ and the maximizer's channel in $\mu$ is garbled in $\sy{\nu}$. Combining these two conditions yields the desired equality for some $\kappa^i \in \sy{K^i}$:
\begin{align*}
\kappa^{\sy{1}}\sy{\nu} = \kappa^{\sy{2}}\mu\sy{.}\nonumber
\end{align*}
\end{proof}

The following is our main result.

\sy{\begin{theorem}\label{Infinite}
Take fixed $\mathbb{X}, \zeta$, fixed and compact $\mathbb{Y}^1, \mathbb{Y}^2$, and information structures $\sy{\nu}$ and $\mu$ which both satisfy Assumption 2.1. Then $\mu$ is better for the maximizer than $\sy{\nu}$ ($\sy{\nu} \lesssim \mu$) over all games in $\Tilde{\mathbb{G}}_\zeta(\mathbb{X}, \mathbb{Y}^1, \mathbb{Y}^2)$ if and only if there exist kernels $\kappa^i \in K^i$ such that:
\begin{align*}
    \kappa^{1}\sy{\nu} = \kappa^{2}\mu\sy{.}\nonumber
\end{align*}
\end{theorem}}
\begin{proof}
The \textit{if} direction follows directly from Lemma \ref{ForwardLemma}:
\begin{align*}
    \sy{\nu} \lesssim \kappa^{\sy{1}}\sy{\nu} = \kappa^{\sy{2}}\mu \lesssim \mu\nonumber
\end{align*}
The \textit{only if} direction is given in Lemma \ref{ConverseLemma}.
\end{proof}

\begin{corollary}\label{Corollary} Take fixed $\mathbb{X}, \zeta$, and fixed and compact $\mathbb{Y}^1, \mathbb{Y}^2$. The value of additional information to a decision maker is never negative for that decision maker in any zero-sum game in $\Tilde{\mathbb{G}}_\zeta(\mathbb{X}, \mathbb{Y}^1, \mathbb{Y}^2)$.
\end{corollary}

\begin{proof}
If $\mu$ is an information structure which is more informative for the maximizer than another information structure $\sy{\nu}$, we know there exists a kernel $\kappa^{\sy{2}}$ such that $\sy{\nu} = \kappa^{\sy{2}}\mu$, and so know that $V^{*}(G, \mu) \geq V^{*}(G, \sy{\nu})$, and $\kappa^{\sy{2}}$ is a well-defined map since we can map any additional information to a fixed number.
\end{proof}

We note that the value of information is not always positive to a player, since many situations (such as where the action set is a singleton) will result in no change in performance despite additional information. 

Corollary \ref{Corollary} is consistent with the work of De Meyer, Lehrer and Rosenberg \cite[Theorem 3.1]{Lehrer2009}, who found this result when studying the value of information in zero-sum games with incomplete information with a slightly different setup, where the `state of nature' was replaced by an individual `type' for each player drawn from a finite space, and where the cost function depended on both players' types. 

We note that for Theorem \ref{Infinite}, the proof will follow for any class of zero-sum games for which every game has an equilibrium solution and Blackwell's ordering of information structures holds for each player when holding the other player's action constant. Therefore, the ordering result can be generalized to be applicable for more general classes of zero-sum games than  $\Tilde{\mathbb{G}}_\zeta(\mathbb{X}, \mathbb{Y}^1, \mathbb{Y}^2)$. 

\subsection{Discussion}

This main result comes with the following intuitive interpretation:

An information structure $\mu$ is better for the maximizer than $\sy{\nu}$ if and only if one of the following holds:

\begin{itemize}
\item[1.] $\sy{\nu}^{\sy{2}}$ is a non-identity garbling of the maximizer's channel from $\mu$, and the minimizer's channel is identical.
\item[2.] $\mu^{\sy{1}}$ is a non-identity garbling of the minimizer's channel from $\sy{\nu}$, and the maximizer's channel is identical.
\item[3.] $\sy{\nu}^{\sy{2}}$ is a non-identity garbling of the maximizer's channel from $\mu$, and $\mu^{\sy{1}}$ is a non-identity garbling of the minimizer's channel from $\sy{\nu}$.
\item[4.] The information structures are identical.
\end{itemize}

In plain terms, this has the following interpretation: In zero-sum games, improving or hurting both players' information structures will never give a general benefit to either player over all games. The only time a player will not do worse under a new information structure is if it only makes his channel better, only makes his opponent's channel worse, makes his channel better and his opponent's channel worse, or is identical to the previous information structure (and the player is guaranteed to not do worse if any of these conditions holds).

In the following, we present an example showing that we cannot {\it view garbling from decision maker to decision maker in isolation from the entire information structure}.  Consider a finite probability space game with $\mathbb{X} = \mathbb{U} = \mathbb{Y}^1 = \mathbb{Y}^2 = \{1,2,3,4\}$, with $x$ distributed according to the uniform distribution, and cost function:
$$
c(x,u^1,u^2) = \begin{cases} -12, \quad u^1 = x \quad \textnormal{and} \quad u^1 \neq u^2 \\ -5, \quad u^1 = x \quad \textnormal{and} \quad u^1 = u^2 \\ 0, \quad \textnormal{otherwise}\end{cases}
$$

Player 1 (the minimizer) gets rewarded for guessing $x$ correctly, and Player 2 (the maximizer) can only limit his losses by playing the same action as Player 1. We can observe that Player 1's optimal strategy will always be to attempt to guess $x$ correctly, since she is only penalized for guessing incorrectly, while Player 2's optimal strategy will always be to attempt to copy Player 1's action since that is the only way he can positively affect the outcome for himself. 

Now consider the following two information structures:

$\mu_1$: Under this information structure, both players receive \textit{the same random measurement} $y^1 = y^2 = y$, where $y = x$ with probability $0.9$ and $y$ is any of the other three incorrect values with probability $0.1/3$. Under this information structure, the best strategy for Player 1 (and thus also for Player 2) is to guess her observation, so $u^1 = u^2 = y$ and the expected payoff is \sy{$-5(0.9) = -4.5$}.

$\mu_2$: Under this information structure, both players receive \textit{conditionally (given $x$) independent measurements}. For Player 1, $y^1 = x$ with probability $0.85$, and is any of the three incorrect values of $x$ with probability $0.05$ each. Player 2 has the same structure as under $\mu_1$, with a $0.9$ chance of success, albeit now uncoupled with Player 1's chance of success. The optimal strategies remain the same under this information structure, but the expected payoff is now \sy{$ -5(0.9)(0.85) + (-12)(0.85)(0.1)= -4.845$}

Therefore, $\mu_2$ is better for the minimizer than $\mu_1$. But, we can observe that the minimizer's channel in $\mu_2$ is garbled from $\mu_1$, in the sense that the distribution on $\mathbb{Y}^1$ for Player 1 can be run through a stochastic kernel to get the distribution under $\mu_2$. The maximizer's channel is identical in both games in the sense that the distribution on $\mathbb{Y}^2$ is unchanged. Yet, the ordering of information structures rule from Theorem \ref{PeskiTheorem} \sy{appears to have been} violated, since the minimizer performs better under the garbled information structure. This demonstrates that we cannot consider garbling in isolation and the comparison should be in view of the entire information structure. 

\sy{While $\mu_2$ appears to be a garbling of $\mu_1$ for the minimizer, it is not a garbling in the sense of this paper, since $\mu_1$ features dependent measurements between the players, while $\mu_2$ has independent measurements between the players. Definition \ref{garbling2} specifies that garblings are done in view of the entire information structure, and so a garbling could not decouple dependence when going from $\mu_1$ to $\mu_2$. If the garbling had been done in accordance with the results of this paper so that $y^1=y^2$ but $y^1$ is then garbled to arrive at some $\tilde{y}^1$ whose probability measure is as specified under $\mu_2$, then the players' measurements would still contain dependence after the garbling. Under this construction, naturally Player 1 would perform worse in the equilibrium under the garbled information structure, since Player 2 has maintained a good ability to copy Player 1's actions when Player 1 is correct due to the dependence being maintained, while Player 1 has received a disadvantage in being able to accurately guess $x$. If the stochastic kernel garbling Player 1's information is as given by $\tilde{\kappa}$ below, where the $(i,j)$ entry is the probability of Player 1 measuring $\tilde{y}^1 = i$ given that the players originally measured $y^1 = y^2 = j$, then the expected equilibrium payoff in this situation would be $-5(0.9)(0.9423) + (-12)(0.1)(0.0192)= -4.263$, which is worse for Player 1, as expected. Under this garbling, Player 1 has a probability of 0.85 of observing the correct measurement $\tilde{y}^1 = x$ and a 0.05 probability of observing any of the three incorrect measurements, matching the distribution specified in the definition of $\mu_2$. 
\begin{equation}
\tilde{\kappa} = \begin{bmatrix} 0.9423 & 0.0192 & 0.0192 & 0.0192 \\
0.0192 & 0.9423 & 0.0192 & 0.0192 \\
0.0192 & 0.0192 & 0.9423 & 0.0192 \\
0.0192 & 0.0192 & 0.0192 & 0.9423 \end{bmatrix} \nonumber
\end{equation}
}

\section{Conclusion}\label{conclusions}
In this paper, we presented an ordering of information structures for a broad class of zero-sum Bayesian games with incomplete information in standard Borel spaces. We also provided two key supporting results: i) a refinement on the conditions for the existence of equilibria in zero-sum games with incomplete information in standard Borel measurement and action spaces and ii) a partial converse to Blackwell's ordering of information structures in this general setting. 

\section{Acknowledgements}

The authors are grateful to M. P\k{e}ski for his detailed comments on both the presentation, technical content, and the related literature, and T. Ba\c{s}ar, M. Le Treust and M. Raginsky for their technical feedback. The authors are also thankful to the three referees who have provided very detailed feedback which has contributed to both the presentation and technical contents of our paper.

\end{document}